\documentclass[11pt]{amsart}
\usepackage[foot]{amsaddr}
\usepackage{amsmath}
\usepackage{amsthm}
\usepackage{amssymb}
\usepackage{mathrsfs}
\usepackage{verbatim}
\usepackage{xcolor}
\usepackage[colorlinks]{hyperref}
\usepackage{enumitem}
\usepackage{multicol}

\newtheorem{lemma}{Lemma}
\newtheorem{theorem}{Theorem}
\newtheorem{definition}{Definition}
\newtheorem{corollary}{Corollary}
\newtheorem{remark}{Remark}
\newtheorem{example}{Example}

\numberwithin{equation}{section}

\allowdisplaybreaks


\begin{document}

\title[BBM-Domains]{Bourgain-Brezis-Mironescu Domains}

\author{Kaushik Bal}

\email{kaushik@iitk.ac.in}

\author{Kaushik Mohanta}
\email{kmohanta@iitk.ac.in}
\author{Prosenjit Roy}

\email{prosenjit@iitk.ac.in}
\address{Indian Institute of Technology Kanpur, India}

\maketitle

\smallskip
\begin{abstract}
 Bourgain \textit{et al.}(2001) proved that for $p>1$ and smooth bounded domain $\Omega\subseteq\mathbb{R}^N$, \begin{equation*}
\lim\limits_{s\to1}(1-s)\iint \limits_{\Omega \times \Omega}\frac{\lvert f(x)-f(y) \rvert^p}{\lvert x-y \rvert^{N+sp}}dx dy=\kappa \int \limits_{\Omega}\lvert \nabla f(x) \rvert^p dx
\end{equation*} for all $f\in L^p(\Omega)$. This gives a characterization of $W^{1,p}(\Omega)$ by means of $W^{s,p}(\Omega)$ seminorms only. For the case $p=1$, D\'avila(2002) proved that when $\Omega$ is a bounded domain with Lipschitz boundary, \begin{equation*}
\lim\limits_{s\to1}(1-s)\iint \limits_{\Omega \times \Omega}\frac{\lvert f(x)-f(y) \rvert}{\lvert x-y \rvert^{N+s}}dx dy=\kappa [f]_{BV(\Omega)}
\end{equation*} for all $f\in L^1(\Omega)$. This characterizes $BV(\Omega)$ in terms of $W^{s,1}(\Omega)$ seminorm.
In this paper we extend the first result and partially extend the second result to extension domains.
\end{abstract}
.\smallskip

\subjclass{\textit{Subject Classification:} {26A33; 46E35; 49J52; 54C35.}
\date{}
\section{Introduction}
For an open set $\Omega\subseteq\mathbb{R}^N$, $0<s<1$ and $1\leq p <\infty$ the Sobolev spaces $W^{1,p}(\Omega)$ and $W^{s,p}(\Omega)$ are defined as:
\begin{align*}
W^{1,p}(\Omega)&=\left\{f\in L^p(\Omega) \ \big| \ \int \limits_{\Omega}\lvert \nabla f(x) \rvert^p dx <\infty\right\},\\
W^{s,p}(\Omega)&=\left\{f\in L^p(\Omega) \ \big| \ \iint \limits_{\Omega \times \Omega}\frac{\lvert f(x)-f(y) \rvert^p}{\lvert x-y \rvert^{N+sp}}dx dy<\infty\right\},
\end{align*}
where $\nabla f$ is understood in the weak sense. These spaces are equipped with the Gagliardo seminorms $[f]_{W^{1,p}(\Omega)}^\frac{1}{p}$, $[f]_{W^{s,p}(\Omega)}^\frac{1}{p}$ given by:
\begin{align*}
[f]_{W^{1,p}(\Omega)}:= \int \limits_{\Omega}\lvert \nabla f(x) \rvert^p dx,\
[f]_{W^{s,p}(\Omega)}:= \iint \limits_{\Omega \times \Omega}\frac{\lvert f(x)-f(y) \rvert^p}{\lvert x-y \rvert^{N+sp}}dx dy.
\end{align*}
For $f\in L^1(\Omega)$, define $$[f]_{BV(\Omega)}:=sup\{\int\limits_{\Omega}f\mbox{div}\phi \ | \ \phi\in C_c^1(\Omega,\mathbb{R}^N), \ \lvert\phi(x)\rvert \leq 1 \ \mbox{for} \ x\in\Omega \}$$
and  the space of bounded variations is defined as
$$BV(\Omega):=\{f\in L^1(\Omega) \ | \ [f]_{BV(\Omega)}<\infty \}.$$
It is well known that for $f\in W^{1,1}(\Omega)$, $[f]_{W^{1,1}(\Omega)}=[f]_{BV(\Omega)}$ (see \cite{giusti}).
\smallskip

Let $\Omega$ be a smooth bounded domain in $ \mathbb{R}^N$. In the celebrated work of Bourgain, Brezis and Mironescu  \cite{bbm}, it is proved that for a family of radial mollifiers $\{\rho_\varepsilon\}$,
\begin{equation}\label{brezis}
\lim\limits_{\varepsilon \to 0}\iint \limits_{\Omega \times \Omega}\frac{\lvert f(x)-f(y) \rvert^p}{\lvert x-y \rvert^p}\rho_\varepsilon(x-y)dxdy=\kappa\int\limits_{\Omega}\lvert \nabla f(x) \rvert^p dx
\end{equation}
for any $f\in L^p(\Omega)$, $1\leq p<\infty$ with \begin{equation}\label{kappa}
\kappa=\int\limits_{\omega\in S^{N-1}}\frac{\lvert \omega\cdot e \rvert^p}{\sigma_{N-1}} d\mathcal{L}^{N-1}
\end{equation}
where $\sigma_{N-1}=\mathcal{L}^{N-1}(S^{N-1})$. In \eqref{brezis} for $f\in L^p(\Omega)\setminus W^{1,p}(\Omega)$, the right hand side is understood to be $\infty$.
\smallskip
For the particular choice
\begin{equation}\label{rho}
   \rho_\varepsilon(x):=\frac{\varepsilon p}{\sigma_{N-1} R^{\varepsilon p} \lvert x \rvert^{N-\varepsilon p}} \chi_{[0,R)}{(\lvert x \rvert)},
   \end{equation}
where $R>0$ is such that $\Omega \subseteq B_R(0)$, one obtains \begin{equation}\label{seminorm-convergance}
\lim \limits_{s \rightarrow 1-}(1-s)[f]_{W^{s,p}(\Omega)}=\kappa\int \limits_{\Omega} \lvert \nabla f(x) \rvert^p dx.
\end{equation}
Bourgain \textit{et al.} in \cite{bbm} also studied the case $p=1$ separately and an easy consequence of their result is: for a smooth bounded domain $\Omega\subseteq\mathbb{R}^N$ and $f\in L^1(\Omega)$ if
\begin{equation}\label{p=1-brezis}
\liminf\limits_{s\to1-}(1-s)[f]_{W^{s,1}(\Omega)}<\infty
\end{equation}
then $f\in BV(\Omega)$. Later D\'avila \cite{davila} proved that for Lipschitz and bounded domain $\Omega$, and $f\in BV(\Omega)$
\begin{equation}\label{p=1-seminorm}
\lim \limits_{s \rightarrow 1-}(1-s)[f]_{W^{s,1}(\Omega)}=\kappa[f]_{BV(\Omega)}.
\end{equation}
Both \eqref{brezis} and \eqref{seminorm-convergance} provide  alternative characterizations of $W^{1,p}(\Omega)$ without using the concept of weak derivatives, whereas \eqref{p=1-brezis} gives a characterization of $BV(\Omega)$. Generalisation of \eqref{p=1-seminorm} to fractional Orlicz setting is addressed in \cite{bonder}.
\smallskip

A natural question that occurs at this point is whether the results of Bourgain \textit{et al.} \cite{bbm} and that of D\'avila \cite{davila} hold for general domains.
Consider the non-extension domain (see Definition \ref{dfext}) given by \begin{equation*}
\Omega=\left(\mathbb{R}^2 \setminus \{ (x_1,x_2) \ | \ x_1\leq 0 \   \mbox{and} \ \lvert x_2\rvert \leq \lvert x_1 \rvert^7 \}\right) \cap B_1(0).
\end{equation*}
Define $f(x):=r(x)\theta(x)$. Then $f\in W^{1,2}(\Omega)\setminus W^{s,2}(\Omega)$ for all $s>\frac{3}{4}$ which implies that the relation \eqref{seminorm-convergance} cannot hold for such a domain $\Omega$. For a detailed discussion on this example, we refer to Di Nezza \textit{et al.} \cite{hhg}. This suggests that the smoothness condition on the domain cannot be completely discarded. Another important observation is that although  \eqref{brezis} in Bourgain \textit{et al.} \cite{bbm} is stated for smooth bounded domains, the same proof goes through for bounded extension domains. This is because smoothness of the boundary was essentially used to get a $W^{1,p}(\mathbb{R}^N)$-extension of $W^{1,p}(\Omega)$ functions.
To the best of our knowledge, the only result that deals with unbounded domains is due to Brezis \cite{constant} where it is proved that \eqref{brezis} holds for $\Omega=\mathbb{R}^N$. At this stage we emphasise that for any unbounded domain $\Omega\subset\mathbb{R}^N$ there is no apparent choice of $\rho_\varepsilon$ which would imply \eqref{seminorm-convergance} from \eqref{brezis} (The natural choice of $\rho_\varepsilon$ as in \eqref{rho} with $R=\infty$ fails as it does not remain integrable).
\smallskip

An important work in this direction is by Leoni and Spector \cite{leoni} who proved that for an arbitrary open set $\Omega$ and $f\in W^{1,p}_{loc}(\Omega)$,
\begin{equation*}\label{eqleoni}
\lim\limits_{\lambda \to 0}\lim\limits_{\varepsilon \to 0}\int \limits_{\Omega_\lambda}\Big( \int \limits_{\Omega_\lambda} \left(\frac{\lvert f(x)-f(y) \rvert^p}{\lvert x-y \rvert^p}\right)^q\rho_\varepsilon(x-y)dx\Big)^\frac{1}{q}dy=\kappa\int\limits_{\Omega}\lvert \nabla f(x) \rvert^p dx
\end{equation*} holds where \begin{equation}\label{omegalambda}
\Omega_\lambda:=\{ x \in \Omega \ \vert \  \mbox{dist}(x,\partial \Omega)>\lambda, \lvert x \rvert \lambda< 1 \}.
\end{equation}
Ponce \cite{ponce} showed that for a bounded smooth domain $\Omega$, a continuous function $\omega:[0,\infty)\to[0,\infty)$ and for $f\in L^{p}(\Omega)$
\begin{equation*}
\liminf\limits_{\varepsilon \to 0}\iint \limits_{\Omega \times \Omega}\omega\left(\frac{\lvert f(x)-f(y) \rvert}{\lvert x-y \rvert}\right)\rho_\varepsilon(x-y)dxdy<\infty
\end{equation*}
characterises $W^{1,p}(\Omega)$.
Ferreira \textit{et al.} \cite{ferreira-ana-hasto} studied a similar limit for generalised Orlicz spaces.
One may also refer to \cite{hasto-ana,nguyen,nguyen1} and the references therein for more information about the topic.
\smallskip
As mentioned earlier, we know that \eqref{brezis}, \eqref{seminorm-convergance}, \eqref{p=1-brezis} and \eqref{p=1-seminorm} hold for bounded extension domains. In this work we drop the boundedness condition of the domain and still recover \eqref{seminorm-convergance}, \eqref{p=1-brezis} and partially recover \eqref{p=1-seminorm}. In particular, the following two theorems are our main results:
\begin{theorem}\label{main-combined}
Let $1<p<\infty$, $\Omega$ be any extension domain and $f\in L^p(\Omega)$. Then
\begin{equation*}
\lim \limits_{s \rightarrow 1-}(1-s)[f]_{W^{s,p}(\Omega)}=\kappa[f]_{W^{1,p}(\Omega)},
\end{equation*}
where we use the convention that $\int \limits_{\Omega} \lvert \nabla f(x) \rvert^p dx=\infty$ when $f\notin W^{1,p}(\Omega)$ and $\kappa$ is as defined in \eqref{kappa}.
\end{theorem}
\begin{theorem}\label{p=1-main-combined}
Let $\Omega$ be any extension domain and $f\in W^{1,1}(\Omega)$. Then
\begin{equation*}
\lim \limits_{s \rightarrow 1-}(1-s)[f]_{W^{s,p}(\Omega)}=\kappa[f]_{BV(\Omega)}
\end{equation*} where $\kappa=\kappa(N,1)$ is as defined in \eqref{kappa}.
Moreover, if $f\in L^1(\Omega)$ satisfies
\begin{equation*}
\liminf \limits_{s \rightarrow 1-}(1-s)[f]_{W^{s,p}(\Omega)}<\infty
\end{equation*} then $f\in BV(\Omega)$.
\end{theorem}
\smallskip

The paper is arranged in the following way: In Section 2, we present various notations, definitions and well known results that shall be used for the proof of our main results, that is Theorem \ref{main-combined} and \ref{p=1-main-combined}. In Section 3, we prove several lemmas that are of independent interest. Then as a consequence of this lemmas, we prove Theorem \ref{main}, \ref{converse} and \ref{p=1-converse}. Theorem \ref{main-combined} follows from Theorem \ref{main} and \ref{converse}. Theorem \ref{p=1-main-combined} follows from Theorem \ref{main} and \ref{p=1-converse}.

\section{ Preliminaries and Definitions}
We shall assume the following notations until otherwise mentioned:
\begin{itemize}
\item  $1\leq p <\infty$, $s \in (0,1)$ and $\Omega \subseteq \mathbb{R}^N$ is an open and connected set,
\item  
$C_c^\infty(\Omega)  := \{    f:\Omega\to\mathbb{R}  \ |  \  f \ \textrm{is infinitely differentiable and has compact support} \}, $
\item $\mathcal{D}(\Omega):=\left\{f{\big|}_\Omega \ | \  f \in C_c^\infty(\mathbb{R}^N)\right\}.$
\item $W^{1,p}(\Omega)$ and $W^{s,p}(\Omega)$ shall denote the Sobolev space (see Brezis \cite{b}) and fractional Sobolev space (see Di Nezza \textit{et al.}\cite{hhg}) respectively,
\item $\kappa$  will be as in \eqref{kappa},
\item  $\mathcal{L}^k$ will denote the $k$-dimensional Lebesgue measure,
\item $C>0$ will denote a generic constant that may differ from line to line but is always independent of $s.$
\end{itemize}
\smallskip

Now we introduce some basic definitions and assumptions that shall be required for formulating our main result.
\begin{definition}\emph{\textbf{[Extension Domain]}}\label{dfext}
An open set $\Omega \subseteq \mathbb{R}^N$ is called an extension domain if there exists a constant $C=C(\Omega,N,p)$ such that for any $f\in W^{1,p}(\Omega)$ there exist $\tilde{f}\in W^{1,p}(\mathbb{R}^N)$ satisfying $\tilde{f}\Big{|}_\Omega=f$ and $\lVert \tilde{f} \rVert_{W^{1,p}(\mathbb{R}^N)} \leq C \lVert f \rVert_{W^{1,p}(\Omega)} $.
\end{definition}
\begin{example}
Any domain with bounded Lipschitz boundary is an extension domain(see \cite{adams}).
\end{example}
\begin{definition}\emph{[\textbf{BBM-Property}]}
For $1\leq p \leq \infty$ we say that $f \in W^{1,p}(\Omega)$ satisfy BBM-Property if
$f$ satisfies \eqref{seminorm-convergance}.
\end{definition}

\begin{definition}\emph{[\textbf{BBM-Domain}]}
For $1\leq p \leq \infty$ we say that $\Omega$ is a BBM-domain if any $f \in W^{1,p}(\Omega)$ satisfies the BBM-Property.
\end{definition}
\begin{example}\label{smooth-bounded-bbm}
For $1\leq p < \infty$ every smooth bounded domain in $\mathbb{R}^N$ is a BBM-domain (see Bourgain \textit{et al.} \cite{bbm}).
\end{example}
Now we present some well known results on Sobolev spaces which will be useful for the proof of our main result.
\begin{lemma}\label{lmdensity}
Let $\Omega$ be an extension domain, then $\mathcal{D}(\Omega)$ is dense in $W^{1,p}(\Omega)$.
\end{lemma}

\begin{lemma}\emph{[Proposition 9.3, Brezis \cite{b}]}\label{brezis-book} Let $\Omega$ be an extension domain in $\mathbb{R}^N$. For any $f$ in $W^{1,p}(\Omega)$ denote the extension of $f$ to $\mathbb{R}^N$ by $\tilde{f} \in W^{1,p}(\mathbb{R}^N)$. Then there exists a constant $C=C(N,p,\Omega)$ such that for any $f$ in $W^{1,p}(\Omega)$ and for any $h$, sufficiently near $0$, in $\mathbb{R}^N$, $$\int \limits_{\mathbb{R}^N} \lvert \tilde{f}(x+h)-\tilde{f}(x) \rvert^pdx \leq   C(N,p,\Omega) \lvert h \rvert ^p [f]_{W^{1,p}(\Omega)}.$$

\end{lemma}

The following lemma is slightly modified version of  Proposition 2.2 of Di Nezza \textit{et al.} \cite{hhg} whose proof is presented here for the sake of completeness.
\begin{lemma}\label{lmdom}
Let $1\leq p <\infty$ and $\Omega$ be an extension domain in $\mathbb{R}^N$. Then there is a constant $C=C(N,p,\Omega)$ such that for any $f$ in $W^{1,p}(\Omega)$ and for $\frac{1}{2} \leq s < 1$,
$$ (1-s)[f]_{W^{s,p}(\Omega)} \leq C(N,p,\Omega) \lVert f \rVert_{W^{1,p}(\Omega)}.$$

\end{lemma}
\begin{proof}Since $\Omega$ is an extension domain, $f\in W^{1,p}(\Omega)$ has an $W^{1,p}(\mathbb{R}^N)$-extension, which we shall also denote by $f$. Denote by $A=\Omega \times (\Omega \cap \{\lvert x-y \rvert <1\})$ and $B=\Omega \times (\Omega \cap \{\lvert x-y \rvert >1\}).$ Therefore,
\begin{eqnarray*}
[f]_{W^{s,p}(\Omega)} \leq
C \left( \iint \limits_A \frac{\lvert f(x)-f(y) \rvert^p}{\lvert x-y \rvert^{N+sp}} dx dy+\iint \limits_B \frac{\lvert f(x)-f(y) \rvert^p}{\lvert x-y \rvert^{N+sp}} dx dy \right)=C(I_1+I_2).
\end{eqnarray*}
We use $A \subseteq \Omega \times B_1(x)$ and the change of variable $h=y-x$ to obtain,

\begin{align*}
I_1 \leq & \iint \limits_{\Omega \times B_1(0)} \frac{\lvert f(x+h)-f(x) \rvert^p}{\lvert h \rvert^{N+sp}} dx dh = \iint \limits_{\Omega \times B_1(0)} \frac{\lvert f(x+h)-f(x) \rvert^p}{\lvert h \rvert^p} \frac{dx dh}{\lvert h \rvert^{N+sp-p}}.
\end{align*}
We use Lemma \ref{brezis-book} to get
\begin{align*}
I_1\leq  C [f]_{W^{1,p}(\Omega)}\int \limits_{B_1(0)}\frac{dh}{\lvert h \rvert^{N+sp-p}} = \frac{C [f]_{W^{1,p}(\Omega)}}{p(1-s)}.
\end{align*}
Now we estimate
\begin{align*}
I_2 =& \iint \limits_B \frac{\lvert f(x)-f(y) \rvert^p}{\lvert x-y \rvert^{N+sp}} dx dy \leq  2^{p-1}\left( \iint \limits_B \frac{\lvert f(x) \rvert^p}{\lvert x-y \rvert^{N+sp}} dxdy+\iint \limits_B \frac{\lvert f(y) \rvert^p}{\lvert x-y \rvert^{N+sp}} dxdy\right).
\end{align*}
Observe that the last two integrals are same (change of variable and Fubini's theorem). This implies,
\begin{align*}
I_2 \leq  2^p \int \limits_{\Omega} \left( \int \limits_{\Omega \cap \{\lvert x-y\rvert >1\}} \frac{dy}{\lvert x-y \rvert^{N+sp}}  \right) \lvert f(x) \rvert^p dx \leq C\lVert f \rVert^p_{L^p(\Omega)} \int \limits_{r=1}^\infty \frac{dr}{r^{1+sp}} = C\lVert f \rVert^p_{L^p(\Omega)}.
\end{align*}
The constant term here is independent of $s$ as we have taken $s>\frac{1}{2}$. This proves the lemma.
\end{proof}

\begin{remark}
Note that Lemma \ref{lmdom} actually suggests that $(1-s)$ should be the scaling factor for which $W^{s,p}$-seminorms would converge to $W^{1,p}$-seminorms.
\end{remark}
\begin{remark}
In the above lemma the hypothesis $\frac{1}{2}\leq s<1$ can be replaced by $0<s_0\leq s<1$.
\end{remark}

\section{Main Result}
In this section we prove the main results of this paper, that is Theorem \ref{main-combined} and \ref{p=1-main-combined}. In order to do it, as stated in the last part of introduction, we shall prove Theorem \ref{main}, \ref{converse} and \ref{p=1-converse}. First we state these theorems.  
\begin{theorem}\label{main}
For $1\leq p <\infty$ any extension domain is a BBM-domain.
\end{theorem}
\begin{theorem}\label{converse}
Let $1< p <\infty$ and $\Omega\subseteq\mathbb{R}^N$ be an open set and $f\in L^p(\Omega)$. Assume that
\begin{equation}\label{converse-hypothesis}
\liminf\limits_{s\to1-}(1-s)[f]_{W^{s,p}(\Omega)}<\infty.
\end{equation}
Then $f\in W^{1,p}(\Omega)$.
\end{theorem}
\begin{theorem}\label{p=1-converse}
Let $\Omega\subseteq\mathbb{R}^N$ be an open set and $f\in L^1(\Omega)$. Assume that
\begin{equation}\label{p=1-converse-hypothesis}
\liminf\limits_{s\to1-}(1-s)[f]_{W^{s,1}(\Omega)}<\infty.
\end{equation}
Then $f\in BV(\Omega)$.
\end{theorem}
We shall first prove Theorem \ref{main}, then using it we shall give the proof of Theorem \ref{converse} and then we shall prove Theorem \ref{p=1-converse}. We divide the proof of these theorems into several lemmas which are of independent interest. The next lemma is an adaptation of a result proved in Bourgain \textit{et al.} \cite{bbm}.

\begin{lemma}\label{lmdenseenough}
Let $\Omega$ be an extension domain in $\mathbb{R^N}$. If for any  $f \in \mathcal{D}(\Omega)$, BBM-Property holds, then $\Omega$ is a BBM-domain.
\end{lemma}

We next state an important observation which shall be used multiple times in the sequel.
\begin{lemma}\label{lmdisjoint}
Let $\Omega_1$ and $\Omega_2$ be two disjoint sets and $ f \in C_c^\infty(\mathbb{R}^N)$. Then
\[
\lim \limits_{s \to 1-} (1-s)\iint \limits_{\Omega_1 \times \Omega_2} \frac{\lvert f(x)-f(y) \rvert^p }{\rvert x- y \lvert^{N+sp}} dx dy=0.
\]
\end{lemma}

\begin{lemma}\label{thRN}
$\mathbb{R}^N$ is a BBM-domain.
\end{lemma}
The rest of the lemmas show how a domain can be identified as BBM-domain if it can be approximated by a sequence of BBM-domains.
\begin{lemma}\label{seq}
Suppose $\{\Omega_n\}_n$ is a decreasing sequence of BBM-domains such that $\cap_{n=1}^\infty \Omega_n=\tilde{\Omega}$ where $\tilde{\Omega}\cup E_1=\overline{\Omega}\cup E_2$ for some sets $E_1,E_2$ such that $\mathcal{L}^N(E_1)=0=\mathcal{L}^N(E_2)$. Let $\Omega$ be an extension domain. Then $\Omega$ is a BBM-domain.
\end{lemma}

Whether the complement of closure of finite union of open balls is an extension domain or not is going to be an important question for us. To understand it, observe that when two balls $B_1, B_2$ cut each other or when $\overline{B}_1$ and $\overline{B}_1$ are disjoint, the complement of their union has Lipschitz boundary and hence we get a positive answer to the above question. But when they touch each other, the answer becomes negative. To avoid this kind of pathological cases we approximate finite union of balls by a smooth open set which is contained in the union.
\begin{lemma}\label{lmbdry}
If $\Omega$ is an open subset of $\mathbb{R}^N$ and $\lambda > 0$ be sufficiently small. Then there is a bounded open set $\Omega_\lambda^*$ with smooth boundary such that $\Omega_\lambda \subseteq \Omega_\lambda^* \subseteq \Omega$, where $\Omega_\lambda$ is as in \eqref{omegalambda}.
\end{lemma}

\begin{corollary}\label{smoothapprox}
The above lemma implies that given a bounded open set $\Omega$ and $\varepsilon>0$, we can approximate it by a smooth and bounded open set $\Omega^*\subseteq\Omega$ such that for any $x\in\partial\Omega$, $\mbox{dist}(x,\partial\Omega^*) < \varepsilon$.
\end{corollary}

\begin{corollary}\label{cor3}
Let $\{B_{r_i}(x_i)\}_{i=1}^\infty$ be a countable family of open balls. Then there exists an increasing family $\{ U_n \}$ of smooth and bounded open sets such that for any $n$, $\overline{U}_n \subseteq \cup_{i=1}^n B_{r_i}^n(x_i)$ and $\cup_{i=1}^\infty U_n = \cup_{i=1}^\infty B_{r_i}^n(x_i)$.
\end{corollary}
\begin{proof}[\textbf{Proof of Corollary \ref{cor3}}]
Existence of smooth and bounded family of open sets $\{ U_n \}_n$, such that for any $n$, $\overline{U}_n \subseteq \cup_{i=1}^n B_{r_i}(x_i)$, is clear. Moreover we can choose the family so that for any $x\in\partial(\cup_{i=1}^n B_{r_i}(x_i))$, $\mbox{dist}(x,\partial U_n)<\frac{1}{n}$. This ensures that $\cup_{n=1}^\infty U_n = \cup_{i=1}^\infty B_{r_i}(x_i)$.
\end{proof}

\begin{lemma}\label{lmcomplement}
Let $\Omega$, $\Omega_1$ be open sets in $\mathbb{R}^N$ such that $\Omega_1 \subseteq \Omega$. Suppose $\Omega$ and $\Omega_1$ are BBM-domains. If $\Omega_2:=\Omega \setminus \overline{\Omega}_1$ is an extension domain, then $\Omega_2$ is also a BBM-domain.
\end{lemma}

\begin{proof}[\textbf{Proof of Lemma \ref{lmdenseenough}}]
Let $f \in W^{1,p}(\Omega)$ and $\varepsilon>0$ be fixed. By Lemma \ref{lmdensity} there exists $g \in \mathcal{D}(\Omega)$ such that
\begin{equation}\label{1}
\lVert g-f \rVert_{W^{1,p}(\Omega)} < \varepsilon.
\end{equation}
Therefore for $s$, sufficiently close to $1$, one has
\begin{equation}\label{2}
\left\lvert (1-s)[g]_{W^{s,p}(\Omega)}-\kappa[g]_{W^{1,p}(\Omega)} \right\rvert <\varepsilon.
\end{equation}
Using triangle inequality,
\begin{multline*}
\left\lvert (1-s)[f]_{W^{s,p}(\Omega)}-\kappa[f]_{W^{1,p}(\Omega)} \right\rvert \leq  (1-s)\left\lvert [f]_{W^{s,p}(\Omega)} - [g]_{W^{s,p}(\Omega)} \right\rvert \\
+ \left\lvert (1-s)[g]_{W^{s,p}(\Omega)}-\kappa[g]_{W^{1,p}(\Omega)} \right\rvert + \kappa\left\lvert [g]_{W^{1,p}(\Omega)}-[f]_{W^{1,p}(\Omega)}\right\rvert\\
\leq(1-s)[f-g]_{W^{s,p}(\Omega)} + \lvert (1-s)[g]_{W^{s,p}(\Omega)}-\kappa[g]_{W^{1,p}(\Omega)} \rvert+ \kappa[f-g]_{W^{1,p}(\Omega)}.
\end{multline*}
Using Lemma \ref{lmdom} we obtain,
\begin{multline*}
\left\lvert (1-s)[f]_{W^{s,p}(\Omega)}-\kappa[f]_{W^{1,p}(\Omega)} \right\rvert
 \leq  C\lVert f-g \rVert_{W^{1,p}(\Omega)}+\left\lvert (1-s)[g]_{W^{s,p}(\Omega)}-\kappa[g]_{W^{1,p}(\Omega)} \right\rvert.
\end{multline*}
Finally using \eqref{1} and \eqref{2} the lemma follows.
\end{proof}

\begin{proof}[\textbf{Proof of Lemma \ref{lmdisjoint}}]
Assume that $f$ is supported in $B_{R-1}(0)$ for some $R>1$.
Set $B_1=\Omega_1 \cap B_R(0)$, $U_1=\Omega_1 \setminus B_R(0)$, $B_2=\Omega_2 \cap B_R(0)$,  and  $U_2=\Omega_2 \setminus B_R(0)$. Then

\begin{multline*}
I :=\iint \limits_{\Omega_1 \times \Omega_2} \frac{\lvert f(x)-f(y) \rvert^p }{\rvert x- y \lvert^{N+sp}}dx dy=\iint \limits_{U_1 \times B_2} \frac{\lvert f(x)-f(y) \rvert^p }{\rvert x- y \lvert^{N+sp}}dx dy+\iint \limits_{B_1 \times U_2} \frac{\lvert f(x)-f(y) \rvert^p }{\rvert x- y \lvert^{N+sp}}dx dy\\
+\iint \limits_{B_1 \times B_2} \frac{\lvert f(x)-f(y) \rvert^p }{\rvert x- y \lvert^{N+sp}}dxdy  =I_1+I_2+I_3.
\end{multline*}
The calculation for $I_2$ will be  similar to that of $I_1$, hence we present the estimate for the term $I_1$ only. Using Fubini's theorem  
\begin{multline*}
I_1= \int \limits_{x \in B_2}\lvert f(x) \rvert^p \Big( \int \limits_{y \in U_1} \frac{dy}{\rvert x- y \lvert^{N+sp}} \Big) dx \leq \int \limits_{x \in B_{R-1}(0)}\lvert f(x) \rvert^p \Big( \int \limits_{y \in  B_{R}(0)^c} \frac{dy}{\rvert x- y \lvert^{N+sp}} \Big) dx\\
\leq  \lVert f \rVert^p_{L^\infty}C\int \limits_{x \in B_{R-1}(0)} \int \limits_{r=1}^\infty r^{-1-sp}dr dx = \frac{C}{s}.
\end{multline*}
Using  the Lipschitz property of $f$  and denoting by $d(y) := \mbox{dist}(y, \partial\Omega_1)$, we estimate the term $I_3$.
\begin{multline*}
I_3  \leq C \int \limits_{y\in B_1}  \Big( \int \limits_{d(y)<\lvert x-y \rvert <2R } \lvert x-y \rvert^{p-N-sp} dx \Big) dy = C \int \limits_{y\in B_1} \Big(\int \limits_{d(y)<\lvert x\rvert <2R} \lvert x\rvert^{p-N-sp} dx \Big)dy\\
= C \int \limits_{y\in B_1} \int \limits_{r=d(y)}^{2R} r^{p-sp-1} dr dy = \frac{C }{1-s} \int \limits_{y\in B_1} \Big[(2R)^{p-sp}-d(y)^{p-sp} \Big] dy.
\end{multline*}
By Dominated convergence theorem (DCT), with a large constant as dominating function, we obtain as $s \to 1,$
$$\int \limits_{y\in B_1} \Big[(2R)^{p-sp}-d(y)^{p-sp} \Big] dy \to 0. $$
Therefore, $\lim \limits_{s \to 1-} (1-s)I_i=0$ for $i=1,2,3$, which finishes the proof of the lemma.
\end{proof}

\begin{proof}[\textbf{Proof of Lemma \ref{thRN}}]
Let  $f \in C_c^\infty(\mathbb{R}^N)$ be such that support of $f$ lies in $B_{R}(0)$. Using Example \ref{smooth-bounded-bbm} we have
\begin{multline*}
\lim \limits_{s \to 1-}(1-s)[f]_{W^{s,p}(\mathbb{R}^N)}
=\lim \limits_{s \to 1-}(1-s) [f]_{W^{s,p}(B_R(0))}\\
+2 \lim \limits_{s \to 1-}(1-s) \iint \limits_{B_R(0) \times \mathbb{R}^N \setminus B_R(0)}\frac{\lvert f(x)-f(y) \rvert^p}{\lvert x-y \rvert^{N+sp}}dx dy\\
=\kappa[f]_{W^{1,p}(B_R(0))}
+2 \lim \limits_{s \to 1-}(1-s) \iint \limits_{B_R(0) \times \mathbb{R}^N \setminus B_R(0)}\frac{\lvert f(x)-f(y) \rvert^p}{\lvert x-y \rvert^{N+sp}} dx dy.
\end{multline*}
By Lemma \ref{lmdisjoint} the last limit is $0$ and the result follows using Lemma \ref{lmdenseenough}.
\end{proof}

\begin{proof}[\textbf{Proof of Lemma \ref{seq}}]
Let $\varepsilon >0$ be given. Let $f \in \mathcal{D}(\Omega)$, that is we assume $f \in C_c^\infty(\mathbb{R}^N)$ with support in $B_R(0)$ for some $R>0$. Then since, each $\Omega_n$ is a BBM-domain,
\begin{multline*}
\kappa[f]_{W^{1,p}(\Omega_n)}
=\lim \limits_{s \to 1-} (1-s) [f]_{W^{s,p}(\Omega_n)}= \lim \limits_{s \to 1-} (1-s) [f]_{W^{s,p}(\Omega)}\\+\lim \limits_{s \to 1-}2 (1-s) \iint \limits_{\Omega \times (\Omega_n \setminus \Omega)} \frac{\lvert f(x)-f(y) \rvert^p}{\lvert x-y \rvert^{N+sp}} dx dy
+\lim \limits_{s \to 1-} (1-s) [f]_{W^{s,p}(\Omega_n \setminus \Omega)}.
\end{multline*}
By  Lemma \ref{lmdisjoint} we can say that the second term  is  equal to $0$. Hence we get,
\begin{equation}\label{4}
\kappa[f]_{W^{1,p}(\Omega_n)}
= \lim \limits_{s \to 1-} (1-s) [f]_{W^{s,p}(\Omega)} +\lim \limits_{s \to 1-} (1-s) [f]_{W^{s,p}(\Omega_n \setminus \Omega)}.
\end{equation}
 Set $A_n:=\Omega_n \setminus \Omega$ and $E_n:=A_n \cap B_R(0)$. We claim that
\begin{equation*}
\lim\limits_{n \to \infty}\lim\limits_{s \to 1-}(1-s)[f]_{W^{s,p}(A_n)}=0.
\end{equation*}
We use Lemma \ref{lmdisjoint} and the fact that $f$ is $0$ in $A_n \setminus E_n$ to get
\begin{multline*}
\lim\limits_{s \to 1-} (1-s)[f]_{W^{s,p}(A_n)}
=\lim\limits_{s \to 1-} (1-s)\Big( [f]_{W^{s,p}(E_n)}+\\ 2\iint \limits_{(A_n \setminus E_n) \times E_n}\frac{\lvert f(x)-f(y) \rvert^p}{\lvert x-y \rvert^{N+sp}}  dxdy
+[f]_{W^{s,p}(A_n \setminus E_n)} \Big) =\lim\limits_{s \to 1-} (1-s)[f]_{W^{s,p}(E_n)}.
\end{multline*}
Further using that $f$ is a Lipschitz function, we obtain
\begin{multline*}
\lim\limits_{s \to 1-} (1-s)[f]_{W^{s,p}(A_n)} \leq C  \lim\limits_{s \to 1-} (1-s)\iint \limits_{E_n \times E_n}\lvert x-y \rvert^{p-N-sp}dxdy\\
\leq C  \lim\limits_{s \to 1-} (1-s)\iint \limits_{E_n \times B_{2R}(0)}\lvert h\rvert^{p-N-sp}dhdy= C \lim\limits_{s \to 1-} (1-s)\int \limits_{E_n}\int \limits_{r=0}^{2R}r^{-1+p-sp}dhdy\\
= C \lim\limits_{s \to 1-} (1-s)\int \limits_{E_n}\frac{(2R)^{p-sp}}{p-sp}dy=\lim\limits_{s \to 1-} C(2R)^{p-sp}\mathcal{L}^N(E_n)
=  C\mathcal{L}^N(E_n).
\end{multline*}
Now each $E_n$ is bounded and $\cap_n E_n$ is empty set, therefore, taking limit as $n \to \infty$, our claim follows.\smallskip

Now, applying DCT to \eqref{4}, as $n \to \infty$, we have
$$
\kappa[f]_{W^{1,p}(\Omega)}=\lim \limits_{s \to 1-} (1-s) [f]_{W^{s,p}(\Omega)}.
$$
Finally using Lemma \ref{lmdenseenough} the lemma follows.
\end{proof}

\begin{proof}[\textbf{Proof of Lemma \ref{lmbdry}}]
We denote, for a set $X\subseteq\mathbb{R}^N$ and $\varepsilon>0$, $N_\varepsilon(X):=\{ y\in\mathbb{R}^N \ | \ \mbox{dist}(y,X)<\varepsilon \}$. Choose $R>0$ such that $\overline{\Omega}_\frac{\lambda}{16} \subseteq B_R(0)$. Consider the continuous function $d: \overline{B}_R(0) \to \mathbb{R}$ defined by
$$
d(x)=
\begin{cases}

\mbox{dist}(x,\partial \Omega_\frac{\lambda}{2}), &x \in \overline{\Omega}_\frac{\lambda}{2}\\
-\mbox{dist}(x,\partial \Omega_\frac{\lambda}{2}), &x \in \overline{B}_R(0) \setminus \Omega_\frac{\lambda}{2}.
\end{cases}
$$

Fix $\varepsilon >0$ such that $N_\varepsilon(\partial \Omega_\frac{\lambda}{2})$ does not intersect with $\partial \Omega_\frac{\lambda}{4}$ and $\partial \Omega_\lambda$. Hence for any $x \in \partial \Omega_\frac{\lambda}{4} \cup \partial \Omega_\lambda$, $\lvert d(x) \rvert >\varepsilon$. By \textit{Stone-Weierstrass Theorem} (see \cite{rudin}) there is a bounded smooth function $g: \overline{B}_R(0) \to \mathbb{R}$ such that $\lVert d-g \rVert_\infty < \varepsilon$.
\smallskip

Take any $x$ from $g^{-1}(0)$. Then $\mbox{dist}(x,\partial \Omega_\frac{\lambda}{2})=\lvert d(x) \rvert= \lvert d(x)-g(x) \rvert< \varepsilon$. Hence $g^{-1}(x)$ is a subset of $N_\varepsilon (\partial \Omega_\frac{\lambda}{2})$. Also \textit{Sard's Theorem} (see \cite{guillemin}) implies that $0$ is a limit point of regular values of $g$ that is there exist regular values, arbitrarily small. Without loss of generality we can assume that $0$ is a regular point of $g$.
\smallskip

Define $\Omega_\lambda^*=g^{-1}(0,\infty)$. Clearly $\partial \Omega_\lambda^*$ is smooth. We claim that $\Omega_\lambda \subseteq \Omega_\lambda^* \subseteq \Omega_\frac{\lambda}{4} $. In fact for $x\in \Omega_\lambda$, $d(x)>\varepsilon$. So we have $0<d(x)-\varepsilon<g(x)$. So, $x\in\Omega_\lambda^*$.
Now let us choose $y\in \Omega_\lambda^*$. Then $g(y)>0$. So, $-\varepsilon<g(y)-\varepsilon<d(y)$. So, $y\in N_\varepsilon(X)\cup\Omega_\frac{\lambda}{2} \subseteq \Omega_\frac{\lambda}{4}$.
\end{proof}

\begin{proof}[\textbf{Proof of Lemma \ref{lmcomplement}}]
Let $f \in \mathcal{D}(\Omega_2)$, that is we assume $f \in C_c^\infty(\mathbb{R}^N)$ with support in $B_R(0)$ for some $R>0$. By Lemma \ref{lmdisjoint} and using that $\Omega$ and $\Omega_1$ are BBM-domains and ignoring measure zero sets in the integration,
\begin{multline*}
\kappa [f]_{W^{1,p}(\Omega)}=\lim \limits_{s \to 1-} (1-s)[f]_{W^{s,p}(\Omega)}=\lim \limits_{s \to 1-} (1-s)[f]_{W^{s,p}(\Omega_1)}
+\lim \limits_{s \to 1-} (1-s)[f]_{W^{s,p}(\Omega_2)}\\
+2\lim \limits_{s \to 1-} (1-s)\iint \limits_{\Omega_1 \times \Omega_2} \frac{\lvert f(x)-f(y) \rvert^p}{\lvert x-y \rvert^{N+sp}} dx dy=\kappa[f]_{W^{1,p}(\Omega_1)}+\lim \limits_{s \to 1-} (1-s)[f]_{W^{s,p}(\Omega_2)}.
\end{multline*}
This implies,
\begin{equation*}
\kappa [f]_{W^{1,p}(\Omega_2)}=\lim \limits_{s \to 1-} (1-s)[f]_{W^{s,p}(\Omega_2)}.
\end{equation*}
Therefore we have the result using Lemma \ref{lmdenseenough}.
\end{proof}

Finally we present the proofs of our main results.
\begin{proof}[\textbf{Proof of Theorem \ref{main}}]
We know that any open set in $\mathbb{R}^N$ is countable union of open balls. Let $\mathbb{R}^N \setminus \overline{\Omega}= \cup_{i=1}^\infty B_{r_i}(x_i)$.Then, by Corollary \ref{cor3}, there exist an increasing family of smooth and bounded open sets, $\{U_n\}_n$, such that for any $n$, $\overline{U}_n \subseteq \cup_{i=1}^n B_{r_i}(x_i)$ and $\cup_{n=1}^\infty U_n = \cup_{i=1}^\infty B_{r_i}(x_i)=\mathbb{R}^N \setminus \overline{\Omega}$. Consider $\Omega_n=\mathbb{R}^N \setminus \cup_{i=1}^n \overline{U}_i=\mathbb{R}^N \setminus \overline{U}_n$. By Example \ref{smooth-bounded-bbm}, each $U_n$ is a BBM-domain. Now, since $U_n$ has smooth bounded boundary, so does $\Omega_n:=\mathbb{R}^N\setminus\overline{U}_n$. This implies that each $\Omega_n$ is an extension domain and we can then apply Lemma \ref{lmcomplement} to conclude that each $\Omega_n$ is a BBM-domain. Observe that $\cap_{n=1}^\infty\Omega_n$ and $\overline{\Omega}$ differ atmost by a measure zero set. So, we can apply Lemma \ref{seq} to conclude that $\Omega$, being an extension domain, is a BBM-domain.
\end{proof}

\begin{proof}[\textbf{Proof of Theorem \ref{converse}}]
By Corollary \ref{smoothapprox} there exist a smooth and bounded sequence of domains $\Omega_n\subsetneq\Omega$ such that $\cup_n\Omega_n=\Omega$. First note that \eqref{converse-hypothesis} implies that
\begin{equation*}
\liminf\limits_{s\to1-}(1-s)[f]_{W^{s,p}(\Omega_n)}<\infty.
\end{equation*}
We can now apply Bourgain \textit{et al.} \cite{bbm} to conclude that for each $n\in\mathbb{N}$, $f\in W^{1,p}(\Omega_n)$, that is there exist $g_n\in (L^p(\Omega))^N$ such that $g_n=\nabla f{\big|}_{\Omega_n}$. Now observe that $g_n=g_{n+1}{\big|}_{\Omega_n}$. Hence the vector valued function $\nabla f:\mathbb{R}^N\to \mathbb{R}^N$ exists although we are not in a position to conclude whether $\nabla f\in (L^p(\Omega))^N$ or not. Moreover in this case
\begin{equation}\label{equation1}
\lim\limits_{s\to1-}(1-s)[f]_{W^{s,p}(\Omega_n)}=\kappa[f]_{W^{1,p}(\Omega_n)}.
\end{equation}

Now we know that the limit in the left hand side of \eqref{equation1} exists in the extended real line, we use it to get
\begin{multline}\label{equation3}
(1-s)[f]_{W^{s,p}(\Omega_n)}\leq(1-s)[f]_{W^{s,p}(\Omega)}\\
\noindent \implies\liminf\limits_{s\to1-}(1-s)[f]_{W^{s,p}(\Omega_n)}\leq\liminf\limits_{s\to1-}(1-s)[f]_{W^{s,p}(\Omega)}\\
\noindent \implies\lim\limits_{s\to1-}(1-s)[f]_{W^{s,p}(\Omega_n)}\leq\liminf\limits_{s\to1-}(1-s)[f]_{W^{s,p}(\Omega)}.
\end{multline}
In \eqref{equation1} we take limit as $n\to\infty$ and apply monotone convergence theorem in the right hand side to conclude
\begin{equation}\label{equation2}
\lim\limits_{n\to\infty}\lim\limits_{s\to1-}(1-s)[f]_{W^{s,p}(\Omega_n)}=\kappa[f]_{W^{1,p}(\Omega)}.
\end{equation}

Now we can apply \eqref{equation2}, \eqref{equation3} and \eqref{converse-hypothesis} to get
\begin{multline*}
\kappa[f]_{W^{1,p}(\Omega)}=\lim\limits_{n\to\infty}\lim\limits_{s\to1-}(1-s)[f]_{W^{s,p}(\Omega_n)}\leq \limsup\limits_{n\to\infty}\liminf\limits_{s\to1-}(1-s)[f]_{W^{s,p}(\Omega)}\\
=\liminf\limits_{s\to1-}(1-s)[f]_{W^{s,p}(\Omega)}<\infty.
\end{multline*}
Consequently $\nabla f\in \left(L^p(\Omega) \right)^N$ and hence $f\in W^{1,p}(\Omega)$.
The last statement now follows from Theorem \ref{main}.
\end{proof}
\begin{proof}[\textbf{Proof of Theorem \ref{p=1-converse}}]
As above let $\Omega_n\subsetneq\Omega$ be an increasing sequence of smooth and bounded domains such that $\cup_n\Omega_n=\Omega$. \eqref{p=1-converse-hypothesis} implies that
\begin{equation*}
\liminf\limits_{s\to1-}(1-s)[f]_{W^{s,1}(\Omega_n)}<\infty.
\end{equation*}
This along with \eqref{p=1-brezis} implies that $f\in BV(\Omega_n)$. Using \eqref{p=1-seminorm} we get
\begin{equation*}
\lim \limits_{s \to 1-}(1-s)[f]_{W^{s,1}(\Omega_n)}=\kappa[f]_{BV(\Omega_n)}.
\end{equation*}
We now use this with \eqref{p=1-converse-hypothesis} to get
\begin{multline*}
\kappa[f]_{BV(\Omega)}=\lim\limits_{n\to \infty} \kappa[f]_{BV(\Omega_n)}
=\lim\limits_{n\to \infty} \lim \limits_{s \to 1-}(1-s)[f]_{W^{s,1}(\Omega_n)}\\
\leq \liminf \limits_{s \to 1-}(1-s)[f]_{W^{s,1}(\Omega)}<\infty.
\end{multline*}
Hence the result follows.
\end{proof}

\bigskip

\textbf{Acknowledgement}  The work of the second author is supported by CSIR(09/092(0940)/2015-EMR-I).
Research work of third author is funded by Matrix grant  (MTR/2019/000585)   and Inspire grant  (IFA14-MA43) of Department of science and Technology (DST).  We are thankful to  P. H\"ast\"o, G. Leoni and  A. Mallick for sharing their knowledge with us on this topic.

\end{document}